\documentclass[pdflatex,sn-mathphys-num]{sn-jnl}

\usepackage{graphicx}
\usepackage{amsmath,amssymb,amsfonts}
\usepackage{amsthm}
\usepackage{mathrsfs}
\usepackage{enumitem}
\usepackage[title]{appendix}
\usepackage{xcolor}
\usepackage{manyfoot}
\usepackage{booktabs}
\usepackage{listings}
\usepackage{lineno}

\theoremstyle{thmstyleone}
\newtheorem{theorem}{Theorem}
\newtheorem{lemma}[theorem]{Lemma}

\theoremstyle{thmstyletwo}

\newtheorem{remark}[theorem]{Remark}

\theoremstyle{thmstylethree}

\newcommand{\qbox}[1]{\quad\hbox{#1}\quad}

\newcommand{\R}{\mathbb{R}}
\newcommand{\X}{\mathcal{X}}
\newcommand{\Y}{\mathcal{Y}}

\begin{document}

\title{ Inertial Primal Dual Dynamics with Hessian-driven Damping for Saddle Point Problems }

\author{\fnm{Zepeng} \sur{Wang}}\email{zepeng.wang@rug.nl}

\author{\fnm{Juan} \sur{Peypouquet}}\email{j.g.peypouquet@rug.nl}

\affil{\orgdiv{Bernoulli Institute for Mathematics, Computer Science and Artificial Intelligence}, \orgname{University of Groningen}, \state{Groningen}, \country{The Netherlands}}

\abstract{ Featuring Hessian-driven damping, two inertial primal dual dynamical systems are proposed for solving smooth saddle point problems with bilinear coupling. For convex-concave functions, we establish a convergence rate $\mathcal{O}\left( \frac{1}{t^2} \right)$ for the primal dual gap; for strongly convex-strongly concave functions, we obtain an asymptotic rate $\mathcal{O}\left( \frac{1}{t^{\alpha-1}} \right)$ ($\alpha\ge 3$ is the damping parameter) without knowledge of the strong convexity parameters, and an accelerated linear convergence rate when the strong convexity parameters are known. As an application of the proposed inertial systems, we also consider the affinely constrained convex optimization problem, and develop an inertial system with Hessian-driven damping, which complements existing results.}



\maketitle

\section{Introduction}
In this paper, we consider the bilinearly coupled saddle point problem:
\begin{equation}\label{Prob: P}
\min_{x\in\X} \max_{y\in\Y}\mathcal{L}(x,y) = f(x) + \langle Ax,y \rangle - g(y),
\end{equation} 
where $\X$ and $\Y$ are real Hilbert spaces, $f:\X\to\R$ and $g:\Y\to\R$ are convex and continuously differentiable functions, and $A:\X\to\Y$ is a bounded linear operator. Our objective is to design an inertial primal dual dynamical system, based on the contemporary understanding of Nesterov's accelerated gradient method, whose trajectories approach solutions to \eqref{Prob: P} at a fast rate.

\subsection*{Dynamic Models for Nesterov's Accelerated Gradient Method}

Nesterov's Accelerated Gradient Method (AGM) \cite{Nesterov_1983} (see also \cite{Beck_2009,Kim_2016,Park_2023} for some variants) remains a popular algorithm to approximate minimizers of a convex and continuously differentiable function $\phi:\X\to\R$. It can be interpreted as a finite-difference discretization of 
\begin{equation}\label{E: AVD}\tag{AVD}
\ddot{x} + \frac{\alpha}{t}\dot{x} + \nabla \phi(x) = 0,
\end{equation}
with $\alpha\ge 3$, as shown in \cite{Su_2016}. For $\alpha\ge 3$, the function values along the trajectories converge at a rate $\mathcal O(\frac{1}{t^2})$, and every trajectory converges to a minimizer of $\phi$ \cite{Attouch_2018,Ryu_2025}. If $\alpha>3$, the convergence rate is $o(\frac{1}{t^2})$ \cite{May_2017} (see \cite{Attouch_2016} for the algorithmic counterpart). In recent years, the analysis of \eqref{E: AVD} has led to a deeper understanding of the behavior of AGM. In \cite{Attouch_2016_Hessian}, and following ideas from \cite{Alvarez_2002}, the authors propose to include a Hessian-driven damping term to \eqref{E: AVD}, leading to
\begin{equation}\label{E: AVD-H}\tag{AVD-H}
\ddot{x} + \frac{\alpha}{t}\dot{x} + \beta\nabla^2 \phi(x)\dot{x} + \nabla \phi(x) = 0,
\end{equation} 
with $\beta>0$. As observed in \cite{Alvarez_2002,Attouch_2016_Hessian}, this reduces the trajectory oscillations that are typical of \eqref{E: AVD}, and also ensures a rapid convergence of the gradient of $\phi$ to zero. Later, it was shown in \cite{Shi_2022} that a slight modification of \eqref{E: AVD-H}, namely
\begin{equation}\label{E: AVD-HR}\tag{AVD-HR}
\ddot{x} + \frac{\alpha}{t}\dot{x} + \beta\nabla^2 \phi(x)\dot{x} + \left(1+\frac{r}{t}\right)\nabla \phi(x) = 0,
\end{equation} 
with $r>0$, provides a high-resolution model that captures more closely the behavior of AGM (see also \cite{Shi_2024}). For example, if $\phi$ is strongly convex, both the trajectories of \eqref{E: AVD-HR} and the sequences generated by AGM converge linearly, in the sense of the function values \cite{Shi_2024}. This {\it is not} the case for the trajectories of \eqref{E: AVD}, for which a lower bound of $\mathcal O(\frac{1}{t^3})$ was shown in \cite{Su_2016}. Moreover, the discretization of \eqref{E: AVD-HR} that leads to AGM is more intuitive and explanatory than the one linking \eqref{E: AVD} to AGM. For a more thorough account of this topic, the reader is referred to \cite{Wang_2025_AVD}.

Now, if $\phi$ is strongly convex {\it with known parameter} $\mu>0$, the best convergence rates are obtained using an autonomous version of \eqref{E: AVD-H}, namely
\begin{equation}\label{E: HBF-H}\tag{HBF-H}
\ddot{x} + \alpha\dot{x} + \beta\nabla^2 \phi(x)\dot{x} + \gamma\nabla \phi(x) = 0,
\end{equation}
where $\alpha$ is chosen as a function of $\mu$, and  $\gamma>0$ (see \cite{Polyak_1964,Attouch_2000,Siegel_2019,Luo_2022,Aujol_2022} for the limiting case $\beta=0$, and \cite{Shi_2022,Attouch_2022,Wang_2025_AGM} for $\beta>0$). These correspond to Nesterov's Accelerated Gradient Method for Strongly Convex functions (AGM-SC) \cite{Nesterov_2004}.

\subsection*{Coupled Systems for Saddle Point Problems}

To solve problem \eqref{Prob: P}, an inertial primal dual system was proposed in \cite{Zeng_2020}:
\begin{equation}\label{E: bilinear_AVD}
\begin{aligned}
\ddot{x} + \frac{\alpha}{t}\dot{x} + \nabla_x \mathcal{L}(x,y+\theta t \dot{y}) &= 0,\\
\ddot{y} + \frac{\alpha}{t}\dot{y} - \nabla_y \mathcal{L}(x+\theta t\dot{x},y) &= 0.
\end{aligned}
\end{equation}
It was shown that a convergence rate $\mathcal{O}\left( \frac{1}{t^2} \right)$ can be guaranteed for the primal dual gap when $\alpha\ge 3$ and $\theta = \frac{1}{2}$ \cite{Zeng_2020}. This rate can be improved to be $o\left( \frac{1}{t^2} \right)$ when $\alpha>3$ and $\theta\in\left( \frac{1}{\alpha-1}, \frac{1}{2} \right)$ \cite{He_2026_bilinear}. Besides, the trajectory convergence is established in finite-dimensional space \cite{He_2026_bilinear}. Some variants of \eqref{E: bilinear_AVD} can be found in \cite{He_2024_bilinear,Ding_2025}. 

In case $g=\langle b, y \rangle$ with $b\in \Y$, problem \eqref{Prob: P} reduces to the affinely constrained convex optimization problem:
\begin{equation}\label{Prob: min_fx}
\min_{x\in\X} f(x),\quad
\text{subject to } Ax = b.
\end{equation}
Using the augmented Lagrangian $\mathcal{L}_\rho$ defined by 
\begin{equation}\label{E: Lm}
\mathcal{L}_\rho(x,y) = f(x) + \langle y, Ax-b \rangle + \frac{\rho}{2}\| Ax-b \|^2,
\end{equation}
where $\rho>0$, the problem can be solved by inertial primal dual dynamics \cite{Zeng_2023}:
\begin{equation}\label{E: affinely_con_AVD}
\begin{aligned}
\ddot{x} + \frac{\alpha}{t}\dot{x} + \nabla_x\mathcal{L}_\rho( x, y+\theta t \dot{y} ) &= 0,\\
\ddot{y} + \frac{\alpha}{t}\dot{y} - \nabla_y\mathcal{L}_\rho(x+\theta t\dot{x},y ) &= 0,
\end{aligned}
\end{equation} 
which provide a convergence rate $\mathcal{O}\left( \frac{1}{t^2} \right)$ for the primal dual gap, the feasibility measure and the function values, when $\alpha\ge 3$ and $\theta=\frac{1}{2}$ \cite{Zeng_2023,Radu_2021}. This rate can be improved to be $o\left( \frac{1}{t^2} \right)$ when $\alpha>3$ and $\theta\in\left( \frac{1}{\alpha-1}, \frac{1}{2} \right)$ \cite{He_2026_affine_dynamics}. The trajectory convergence was established for $\alpha>3$ in Hilbert space \cite{Radu_2021} and for $\alpha=3$ in finite-dimensional space \cite{He_2026_affine_dynamics}.  A time-scaled system was studied in \cite{Hulett_2023}. Other variants of system \eqref{E: affinely_con_AVD} can be found in \cite{He_2021,He_2022_TAC,He_2022_AUTO,Attouch_2022_ADMM}. These dynamical systems do not contain the full Hessian-driven damping, which may result in undesirable oscillations. Notice also that when discretization of \eqref{E: affinely_con_AVD}, some delicate approximations have to be made so as to obtain a convergent algorithm \cite{Radu_2023,He_2026_affine_algo}. Motivated by smoother trajectories and an intuitive discretization, many efforts have been made to develop an inertial dynamical system with Hessian-driven damping. A mixed-order primal dual system was proposed in \cite{He_2025}, and another system was developed from the perspective of maximally monotone operators \cite{Radu_2025_OGDA}. We notice that neither the dynamics in \cite{He_2025} nor those in \cite{Radu_2025_OGDA} can recover the algorithm in \cite{Radu_2023}. Extension of \eqref{E: affinely_con_AVD} to the Hessian-driven damping case whose direct discretization would give the algorithm in \cite{Radu_2023} is one of motivations of this paper.        

If $f$ and $g$ are strongly convex with known parameters $\mu_f,\mu_g>0$, a heavy-ball primal dual dynamical system was proposed in \cite{He_2025_PD_SC}:     
\begin{equation}\label{E: bilinear_HBF}
\begin{aligned}
\ddot{x} + 2\sqrt{\mu_f}\dot{x} + \nabla_x \mathcal{L}(x,y+\theta \dot{y}) &= 0,\\
\ddot{y} + 2\sqrt{\mu_g}\dot{y} - \nabla_y \mathcal{L}(x+\theta \dot{x},y) &= 0.
\end{aligned}
\end{equation}
Setting $\theta = \max\left( \frac{1}{\sqrt{\mu_f}}, \frac{1}{\sqrt{\mu_g}} \right)$, one can obtain an accelerated convergence rate $\mathcal{O}\left( e^{-\min(\sqrt{\mu_f},\sqrt{\mu_g}) t} \right)$ for the primal dual gap.

\subsection*{Our contribution}

We propose a set of primal dual dynamical systems with Hessian-driven damping, and whose trajectories approach solutions of Problem \eqref{Prob: P}. More precisely,
\begin{itemize}
\item We modify \eqref{E: bilinear_AVD} to include a Hessian-driven damping term in the spirit of \eqref{E: AVD-HR}. If $f$ and $g$ are convex, we establish a convergence rate of $\mathcal O(\frac{1}{\gamma t^2})$ for the Lagrangian primal dual gap $\mathcal{L}(x,y^*) - \mathcal{L}(x^*,y)$ as a measure of optimality.

\item In the strongly convex case (even when the parameters are unknown), the convergence rate becomes $\mathcal O(\frac{1}{\gamma t^{\alpha-1}})$. This result is made possible by inclusion of the Hessian-driven damping. All these results are presented in Section \ref{Sec: PD-AVD-H}.

\item If $f$ and $g$ are strongly convex with known coefficients $\mu_f,\mu_g$, we propose a Hessian-driven variant of \eqref{E: bilinear_HBF}, in line with \eqref{E: HBF-H}, for which the convergence rate is $\mathcal O(e^{\sqrt{\gamma_0\mu}t})$, where $\mu=\min\{\mu_f,\mu_g\}$, and $\gamma_0\sim\gamma$. This is done in Section \ref{Sec: PD-HBF-H}.

\item For affinely constrained optimization problems, we develop a system involving the augmented Lagrangian. This allows us to obtain a convergence rate of $\mathcal O(\frac{1}{\gamma t^2})$ for the augmented Lagrangian gap $\mathcal{L}_\rho\left(x(t),y^*\right) - \mathcal{L}_\rho(x^*,y^*)$, as well as the function values $|f\left( x(t) \right) - f(x^*) |$ and the feasibility gap $\|Ax(t)-b\|$. These results are discussed in Section \ref{Sec: affine convex}.
\end{itemize}

In all that follows, we assume that $f:\X\to\R$ and $g:\Y\to\R$ are convex and twice continuously differentiable. A reduction to the first order in order to deal with the nonsmooth case can be found in \cite{Attouch_2016_Hessian}.

\section{Inertial Dynamics I: The Convex Case}\label{Sec: PD-AVD-H}
In this section, we propose to solve \eqref{Prob: P} using the following inertial primal dual dynamics with Hessian-driven damping:
\begin{equation}\label{System: PD_system_AVD}\tag{PD-AVD-H}
\left\{
\begin{array}{lcl}
\ddot{x} + \frac{\alpha}{t}\dot{x} + \beta_f\nabla_{xx}\mathcal{L}(x,y)\dot{x} + \left( \gamma + \frac{r}{t} \right) \nabla_x \mathcal{L}(x,y+\theta t \dot{y}) &=& 0,\\[4pt]
\ddot{y} + \frac{\alpha}{t}\dot{y} - \beta_g\nabla_{yy}\mathcal{L}(x,y)\dot{y} - \left( \gamma + \frac{r}{t} \right) \nabla_y \mathcal{L}(x+\theta t\dot{x},y) &=& 0,
\end{array} 
\right.
\end{equation} 
where $\alpha\ge 3$ and $\beta_f,\beta_g,\gamma,r,\theta>0$, with $t\ge t_0$ for some $t_0>0$.

Setting $\beta_f=\beta_g=r=0$ and $\gamma=1$, \eqref{System: PD_system_AVD} can recover the systems in \cite{Zeng_2020,Ding_2025,He_2026_bilinear} and corresponds to a special case of the system in \cite{He_2024_bilinear}. But due to its inclusion of Hessian-driven damping $\beta_f\nabla_{xx}\mathcal{L}(x,y)\dot{x}$ and $\beta_g\nabla_{yy}\mathcal{L}(x,y)\dot{y}$, system \eqref{System: PD_system_AVD} is different from those in \cite{Zeng_2020,Ding_2025,He_2024_bilinear,He_2026_bilinear}. Thanks to inclusion of these two terms, as one will see in Theorem \ref{Thm: AVD_SC}, we are able to derive new convergence results in the strongly convex case.

\subsection{The Primal Dual Gap}\label{Subsec: F}
To simplify the notation, we define $s=(x,y)$, and denote the primal dual gap by 
\begin{equation}\label{E: F}
F(s):= \mathcal{L}(x,y^*) - \mathcal{L}(x^*,y) = f(x) - f(x^*) + g(y) - g(y^*) + \langle Ax, y^* \rangle - \langle Ax^*,y\rangle.
\end{equation}
Notice that $F(s)\ge 0$ for all $x\in\X$ and $y\in\Y$, and 
$$ \nabla F(s) = \begin{bmatrix}
\nabla f(x) + A^* y^*\\[2pt]
\nabla g(y) - Ax^*
\end{bmatrix}.$$ 
As a result, \eqref{System: PD_system_AVD} can be rewritten as
\begin{equation}\label{E: PD_system_AVD}
\ddot{s} + \frac{\alpha}{t} \dot{s} + \begin{bmatrix}
\beta_f \nabla^2 f(x) \dot{x}\\[2pt]
\beta_g \nabla^2 g(y) \dot{y}
\end{bmatrix} + \left( \gamma + \frac{r}{t} \right) \nabla F + \left( \gamma + \frac{r}{t} \right) \begin{bmatrix}
A^* ( \theta t \dot{y} + y - y^* ) \\[2pt]
- A( \theta t \dot{x} + x - x^* )
\end{bmatrix} = 0.
\end{equation}

The following result shall be useful in the forthcoming analysis.

\begin{lemma}\label{Lem: dF}
Consider the primal dual gap $F(s)$ defined by \eqref{E: F}. Then, we have
$$ \langle \nabla F, s-s^* \rangle = F + [ \langle \nabla f(x), x-x^* \rangle - \left( f(x) - f(x^*) \right) ]
  + [ \langle \nabla g(y), y-y^* \rangle - \left( g(y) - g(y^*) \right) ]. $$
\end{lemma} 

\begin{proof}
With \eqref{E: F} in mind, we have
\begin{align*}
\langle \nabla F, s-s^* \rangle
&= \left\langle \begin{bmatrix}
\nabla f(x) + A^* y^*\\[2pt]
\nabla g(y) - Ax^*
\end{bmatrix}, \begin{bmatrix}
x-x^*\\[2pt]
y-y^*
\end{bmatrix}  \right\rangle \\
&= \langle \nabla f(x), x-x^* \rangle + \left\langle A^* y^*, x \right\rangle + \langle \nabla g(y), y-y^* \rangle - \left\langle Ax^*, y \right\rangle\\
&= [ \langle \nabla f(x), x-x^* \rangle - \left( f(x) - f(x^*) \right) ] + f(x) - f(x^*) + \left\langle A^* y^*, x \right\rangle \\
&\quad + [ \langle \nabla g(y), y-y^* \rangle - \left( g(y) - g(y^*) \right) ] + g(y) - g(y^*) - \left\langle Ax^*, y \right\rangle \\
&= F + [ \langle \nabla f(x), x-x^* \rangle - \left( f(x) - f(x^*) \right) ] 
  + [ \langle \nabla g(y), y-y^* \rangle - \left( g(y) - g(y^*) \right) ],      
\end{align*}
and the proof is complete.
\end{proof}

\begin{remark}
    In particular, $\langle \nabla F, s-s^* \rangle\ge F\ge 0$.
\end{remark}

\subsection{Energy Estimations}
Our analysis centers around the energy function $W:[t_0,\infty)\to\R$, defined by
\begin{equation}\label{E: W_AVD}
\begin{aligned}
W(t)
&= \frac{1}{2}\| t\dot{s} + (\alpha-1)(s-s^*) \|^2 + t(\gamma t + r)F(s) \\
&\quad + (\alpha-1)\beta_f t \left[\langle \nabla f(x), x-x^* \rangle - \left( f(x) - f(x^*) \right) \right] \\ 
&\quad + (\alpha-1)\beta_g t \left[ \langle \nabla g(y), y-y^* \rangle - \left( g(y) - g(y^*) \right) \right].
\end{aligned}
\end{equation}
Notice that $W(t)\ge 0$ for all $t\ge t_0$, since $f$ and $g$ are convex and $F(s)\ge 0$. 

\begin{remark}
Based on the systems proposed in \cite{Zeng_2020,Ding_2025,He_2024_bilinear,He_2026_bilinear}, it is natural to include in the dynamics the Hessian-driven damping $\beta_f\nabla_{xx}\mathcal{L}(x,y)\dot{x}$ and $\beta_g\nabla_{yy}\mathcal{L}(x,y)\dot{y}$. But at the same time, a convergence proof for the resultant system becomes more difficult. The last two terms in \eqref{E: W_AVD} and Lemma \ref{Lem: dF} are the two key enabling components for the convergence analysis.   
\end{remark}

\begin{remark} \label{R:smooth}
Recall the optimality condtion $\nabla F(s^*)=0$. If $f$ is $L_f$-smooth, then $W(t)\ge \frac{(\alpha-1)\beta_ft}{2L_f}\|\nabla f(x)+A^*y^*\|^2$. If $g$ is $L_g$-smooth, then $W(t)\ge \frac{(\alpha-1)\beta_gt}{2L_g}\|\nabla g(y)-Ax^*\|^2$.    
\end{remark}

Using \eqref{E: W_AVD}, we have the following:

\begin{lemma}\label{Lem: W_bound_AVD}
Let $W$ be defined by \eqref{E: W_AVD}. Set $\theta = \frac{1}{\alpha-1}$. Then, we have
\begin{align*}
\dot{W}(t)
&= - \beta_f t^2 \langle \nabla^2 f(x)\dot{x}, \dot{x} \rangle 
   - \beta_g t^2 \langle \nabla^2 g(y)\dot{y}, \dot{y} \rangle 
   - [ (\alpha-3)\gamma t + (\alpha-2)r ] F \\
&\quad - (\alpha-1)[ \gamma t + r-\beta_f  ] \left[\langle \nabla f(x), x-x^* \rangle - \left( f(x) - f(x^*) \right) \right] \\
&\quad - (\alpha-1)[ \gamma t + r-\beta_g ] \left[ \langle \nabla g(y), y-y^* \rangle - \left( g(y) - g(y^*) \right) \right].
\end{align*}  
\end{lemma}

\begin{proof}
Denote $W(t) := \sum_{i=1}^4 W_i(t)$, with $W_i$ $(i=1,2,3,4)$ being the terms on the right-hand side of \eqref{E: W_AVD} in the order in which they appear. We begin by using \eqref{E: PD_system_AVD}, to obtain
\begin{align*}
\dot{W}_1(t)
&= \langle t\dot{s} + (\alpha-1)(s-s^*), t\ddot{s} + \alpha\dot{s} \rangle \\
&= -\left\langle t\dot{s} + (\alpha-1)(s-s^*), \begin{bmatrix}
   \beta_f t \nabla^2 f(x) \dot{x}\\[2pt]
   \beta_g t \nabla^2 g(y) \dot{y}
   \end{bmatrix} \right\rangle 
 - ( \gamma t + r ) \langle \nabla F, t\dot{s} + (\alpha-1)(s-s^*) \rangle \\
&\quad - ( \gamma t + r ) \left\langle t\dot{s} + (\alpha-1)(s-s^*), \begin{bmatrix}
A^* ( \theta t \dot{y} + y - y^* ) \\[2pt]
- A( \theta t \dot{x} + x - x^* )
\end{bmatrix} \right\rangle.
\end{align*}
Since $\theta = \frac{1}{\alpha-1}$, we have
\begin{align*}
&\quad \left\langle t\dot{s} + (\alpha-1)(s-s^*), \begin{bmatrix}
A^* ( \theta t \dot{y} + y - y^* ) \\[2pt]
- A( \theta t \dot{x} + x - x^* )
\end{bmatrix} \right\rangle \\ 
&= \frac{1}{\theta}\left\langle \begin{bmatrix}
\theta t \dot{x} + x-x^*\\[2pt]
\theta t \dot{y} + y-y^*
\end{bmatrix}, \begin{bmatrix}
A^* ( \theta t \dot{y} + y - y^* ) \\[2pt]
- A( \theta t \dot{x} + x - x^* )
\end{bmatrix} \right\rangle = 0.
\end{align*}
As a result,
\begin{align*}
\dot{W}_1(t)
&= - \beta_f t^2 \langle \nabla^2 f(x)\dot{x}, \dot{x} \rangle 
   - \beta_g t^2 \langle \nabla^2 g(y)\dot{y}, \dot{y} \rangle  
   - (\alpha-1)\beta_f t \langle \nabla^2 f(x)\dot{x}, x-x^* \rangle \\
&\quad - (\alpha-1)\beta_g t \langle \nabla^2 g(y)\dot{x}, y-y^* \rangle 
   - t(\gamma t + r) \langle \nabla F, \dot{s} \rangle 
   - (\alpha-1)(\gamma t + r) \langle \nabla F, s-s^* \rangle. 
\end{align*}
Likewise, we obtain
\begin{align*}
\dot{W}_2(t) &= (2\gamma t + r) F + t(\gamma t + r)\langle \nabla F, \dot{s} \rangle,\\
\dot{W}_3(t) &= (\alpha-1)\beta_f t \langle \nabla^2 f(x) \dot{x}, x-x^* \rangle 
 + (\alpha-1)\beta_f \left[\langle \nabla f(x), x-x^* \rangle - \left( f(x) - f(x^*) \right) \right],\\
\dot{W}_4(t) &= (\alpha-1)\beta_g t \langle \nabla^2 g(y)\dot{y}, y-y^* \rangle 
 + (\alpha-1)\beta_g \left[ \langle \nabla g(y), y-y^* \rangle - \left( g(y) - g(y^*) \right) \right]. 
\end{align*}
Summing up these equations, we obtain
\begin{align*}
\dot{W}(t)
&= - \beta_f t^2 \langle \nabla^2 f(x)\dot{x}, \dot{x} \rangle
   - \beta_g t^2 \langle \nabla^2 g(y)\dot{y}, \dot{y} \rangle 
   + (2\gamma t + r) F 
 - (\alpha-1)(\gamma t + r) \langle \nabla F, s-s^* \rangle \\
&\quad + (\alpha-1)\beta_f \left[\langle \nabla f(x), x-x^* \rangle - \left( f(x) - f(x^*) \right) \right] \\
&\quad + (\alpha-1)\beta_g \left[ \langle \nabla g(y), y-y^* \rangle - \left( g(y) - g(y^*) \right) \right],
\end{align*}
which, combined with Lemma \ref{Lem: dF}, gives the desired result.
\end{proof}

\begin{remark}\label{Rem: W_bound_AVD}
Since $f$ and $g$ are convex, we have
$$\langle \nabla^2 f(x)\dot{x}, \dot{x} \rangle \ge 0, \qbox{} \langle \nabla f(x), x-x^* \rangle - \left( f(x) - f(x^*) \right) \ge 0, $$
and
$$\langle \nabla^2 g(y)\dot{y}, \dot{y} \rangle\ge 0, \qbox{} \langle \nabla g(y), y-y^* \rangle - \left( g(y) - g(y^*) \right) \ge 0. $$
If $\alpha\ge 3$ and $r\ge\max(\beta_f,\beta_g)$, then $W(t)$ is nonincreasing, whence $W(t)\le W(t_0)$ for all $t\ge t_0$.
\end{remark}

\subsection{Convergence Analysis}\label{Subsec: convergence_analysis_PD_AVD-H}

\begin{theorem} \label{T:convex}
Let $f:\X\to\R$ and $g:\Y\to\R$ be convex and continuously differentiable. Let $(x,y):[t_0,\infty)\to\X\times\Y$ be a solution of the system \eqref{System: PD_system_AVD}, where
$$\alpha\ge 3, \qbox{}\beta_f,\beta_g,\gamma>0, \qbox{} r\ge\max(\beta_f,\beta_g),\qbox{and} \theta = \frac{1}{\alpha-1}. $$
Then for every $t\ge t_0$, we have:
\begin{enumerate}[label=(\roman*), itemsep=0.7em, leftmargin=12mm]
\item The trajectory $(x,y)$ is bounded.

\item $\| \dot{x}(t) \|^2 + \| \dot{y}(t) \|^2   = \mathcal{O}\left( \frac{1}{t^2} \right)$. 

\item $\mathcal{L}(x(t),y^*) - \mathcal{L}(x^*,y(t)) =\mathcal{O}\left( \frac{1}{t^2} \right) 
$. 
\end{enumerate} 
\end{theorem}

\begin{proof}
(i) By Remark \ref{Rem: W_bound_AVD}, we have $W(t)\le W(t_0)$ for all $t\ge t_0$. Using \eqref{E: W_AVD}, we obtain
\begin{equation}\label{E: bound_temp}
\| t\dot{s} + (\alpha-1)(s-s^*) \|^2 \le 2W(t) \le 2W(t_0).
\end{equation}
Developing the terms in the left side, we arrive at
$$ t^2\| \dot{s} \|^2 + (\alpha-1)^2 \| s-s^* \|^2 + 2(\alpha-1)t\langle s-s^*, \dot{s} \rangle \le 2W(t_0), $$
so that
$$ (\alpha-1)\| s-s^* \|^2 + 2t \langle s-s^*, \dot{s} \rangle \le \frac{2W(t_0)}{\alpha-1}. $$ 
This is equivalent to
$$ \frac{d}{dt}\left[ t^{\alpha-1}\| s-s^* \|^2 \right] \le \frac{2W(t_0) t^{\alpha-2}}{\alpha-1}. $$
Integrating from $t_0$ to $t>t_0$ gives
$$ t^{\alpha-1}\| s-s^* \|^2 - t_0^{\alpha-1}\| s_0-s^* \|^2 \le \frac{2W(t_0)}{(\alpha-1)^2}\left( t^{\alpha-1} - t_0^{\alpha-1} \right).  $$
As a result,
$$ \| s-s^* \|^2 \le \frac{2W(t_0)}{(\alpha-1)^2} + \| s_0-s^* \|^2 < \infty, $$
which means that the trajectory is bounded. 

(ii) Using \eqref{E: bound_temp}, we obtain
\begin{align*}
    t\| \dot{s} \| & \le \| t\dot{s} + (\alpha-1)(s-s^*) \| + (\alpha-1)\| s-s^* \| \\
    & \le \sqrt{2W(t_0)} + \sqrt{2W(t_0)+(\alpha-1)^2\|s_0-s^*\|^2},
\end{align*}
and the claim follows.

(iii) Since $W(t) \le W(t_0)$ for all $t\ge t_0$, we have
$$ F(s) \le \frac{W(t)}{t(\gamma t + r)} \le \frac{W(t_0)}{t(\gamma t + r)}. $$ 
Recalling \eqref{E: F}, we obtain the desired result.
\end{proof}

\begin{remark}
The trajectory boundedness and the $\mathcal{O}\left(\frac{1}{t}\right)$ decay of the velocity still hold when $\alpha=3$, which complements the result in \cite{Zeng_2020}.
\end{remark}

\begin{remark} \label{R:convergent_gradient}
    Combining Theorem \ref{T:convex} with Remark \ref{R:smooth}, we deduce that if $f$ is smooth, then $\|\nabla f(x(t))+A^*y^*\|^2=\mathcal{O}\left( \frac{1}{t} \right)$. Similarly, if $g$ is smooth, then $\|\nabla g(y(t))-Ax^*\|^2=\mathcal{O}\left( \frac{1}{t} \right)$. This also holds if $f$ (respectively $g$) is {\it locally} smooth, in the sense that $\nabla f$ (respectively $\nabla g$) is Lipschitz continuous on every ball. Thanks to the continuity of the corresponding Hessian, the latter is true under no additional assumptions if $\mathcal X$ (respectively $\mathcal Y$) is finite dimensional.
\end{remark}

\subsection{Adaptivity of the Rate to Strong Convexity}

In what follows, we establish faster convergence rates for \eqref{System: PD_system_AVD} in case $f$ and $g$ are strongly convex, with possibly unknown parameters. To this end, we begin with the following estimation:

\begin{lemma}\label{Lem: W_bound_AVD_SC}
Let $f:\X\to\R$ and $g:\Y\to\R$ be strongly convex with parameters $\mu_f,\mu_g>0$. Let $W$ be defined by \eqref{E: W_AVD} with $\theta = \frac{1}{\alpha-1}$. If $\alpha\ge 3$, then we have
\begin{align*}
&\quad t\dot{W} + (\alpha-3)W \\
&\le -\left[ \min(\mu_f\beta_f,\mu_g\beta_g) t - (\alpha-3) \right] t^2 \| \dot{s} \|^2\\
& -(\alpha-1)t \left[ \gamma t + r-(\alpha-4)\beta_f - \tfrac{2(\alpha-1)(\alpha-3)}{\mu_f t} \right] \left[\langle \nabla f(x), x-x^* \rangle - \left( f(x) - f(x^*) \right) \right] \\
& -(\alpha-1)t \left[ \gamma t + r-(\alpha-4)\beta_g - \tfrac{2(\alpha-1)(\alpha-3)}{\mu_g t} \right] \left[ \langle \nabla g(y), y-y^* \rangle - \left( g(y) - g(y^*) \right) \right]. 
\end{align*}
\end{lemma}

\begin{proof}
Since $f$ and $g$ are strongly convex, we have
$$ \langle \nabla^2 f(x) \dot{x}, \dot{x} \rangle \ge \mu_f \| \dot{x} \|^2,
\qbox{and} \langle \nabla^2 g(y) \dot{y}, \dot{y} \rangle \ge \mu_g \| \dot{y} \|^2, $$
which, in combination with Lemma \ref{Lem: W_bound_AVD}, gives
\begin{equation}\label{E: dev_W_bound_AVD_SC_temp}
\begin{aligned}
\dot{W}
&\le - \min(\mu_f\beta_f,\mu_g\beta_g)t^2 \| \dot{s} \|^2 - (\alpha-3)(\gamma t + r) F \\
&\quad -(\alpha-1)( \gamma t + r-\beta_f ) \left[\langle \nabla f(x), x-x^* \rangle - \left( f(x) - f(x^*) \right) \right] \\
&\quad -(\alpha-1)( \gamma t + r-\beta_g ) \left[ \langle \nabla g(y), y-y^* \rangle - \left( g(y) - g(y^*) \right) \right].
\end{aligned}
\end{equation}
By definition of $W(t)$ in \eqref{E: W_AVD}, we have
\begin{equation}\label{E: W_bound_AVD_SC_temp}
\begin{aligned}
W
&\le t^2 \| \dot{s} \|^2 + (\alpha-1)^2 \| s-s^* \|^2 + t(\gamma t + r) F \\ 
&\quad + (\alpha-1)\beta_f t \left[\langle \nabla f(x), x-x^* \rangle - \left( f(x) - f(x^*) \right) \right] \\
&\quad + (\alpha-1)\beta_g t \left[ \langle \nabla g(y), y-y^* \rangle - \left( g(y) - g(y^*) \right) \right] \\
&\le t^2 \| \dot{s} \|^2 + t(\gamma t + r) F 
 + (\alpha-1) t \left[ \beta_f + \frac{2(\alpha-1)}{\mu_f t} \right] \left[\langle \nabla f(x), x-x^* \rangle - \left( f(x) - f(x^*) \right) \right] \\
&\quad + (\alpha-1) t \left[ \beta_g + \frac{2(\alpha-1)}{\mu_g t} \right] \left[ \langle \nabla g(y), y-y^* \rangle - \left( g(y) - g(y^*) \right) \right],
\end{aligned}
\end{equation}
where the last inequality is due to strong convexity of $f$ and $g$. Combining \eqref{E: dev_W_bound_AVD_SC_temp} and \eqref{E: W_bound_AVD_SC_temp} gives the desired result.
\end{proof}

\begin{remark}\label{Rem: W_bound_AVD_SC}
If $r\ge \max(\beta_f,\beta_g)$ and
$$ t\ge \max\left( \tfrac{\alpha-3}{\min(\mu_f\beta_f,\mu_g\beta_g)}, \tfrac{(\alpha-3)\beta_f}{\gamma} + \sqrt{\tfrac{ 2(\alpha-1)(\alpha-3) }{\gamma\mu_f} }, \tfrac{(\alpha-3)\beta_g}{\gamma} + \sqrt{\tfrac{ 2(\alpha-1)(\alpha-3) }{\gamma\mu_g} } \right), $$
we have $t\dot{W} + (\alpha-3)W \le 0$.
\end{remark}

Now we are in a position to establish the convergence result under strong convexity of $f$ and $g$ for system \eqref{System: PD_system_AVD}.

\begin{theorem}\label{Thm: AVD_SC}
Let $f:\X\to\R$ and $g:\Y\to\R$ be strongly convex with parameters $\mu_f,\mu_g>0$ and continuously differentiable. Let $(x,y):[t_0,\infty)\to\X\times\Y$ be a solution of the system \eqref{System: PD_system_AVD}, where
$$\alpha\ge3, \qbox{}\beta_f,\beta_g,\gamma>0, \qbox{} r\ge\max(\beta_f,\beta_g), \qbox{and} \theta = \frac{1}{\alpha-1}. $$
Let
$$T=\max\left( t_0, \tfrac{\alpha-3}{\min(\mu_f\beta_f,\mu_g\beta_g)}, \tfrac{(\alpha-3)\beta_f}{\gamma} + \sqrt{\tfrac{ 2(\alpha-1)(\alpha-3) }{\gamma\mu_f} }, \tfrac{(\alpha-3)\beta_g}{\gamma} + \sqrt{\tfrac{ 2(\alpha-1)(\alpha-3) }{\gamma\mu_g} } \right).$$
Then for every $t\ge T$, we have
\begin{enumerate}[label=(\roman*), itemsep=0.7em, leftmargin=12mm]
\item $\sqrt{\| x(t) - x^* \|^2 + \| y(t) - y^* \|^2} = \mathcal{O}\left( \frac{1}{t^{(\alpha-2)/2}} \right)$.

\item $\sqrt{ \| \dot{x}(t) \|^2 + \| \dot{y}(t) \|^2 } = \mathcal{O}\left( \frac{1}{t^{(\alpha-1)/2}} \right)$.

\item $\mathcal{L}(x(t),y^*) - \mathcal{L}(x^*,y(t)) = \mathcal{O}\left( \frac{1}{t^{\alpha-1}} \right)$.
\end{enumerate}
\end{theorem}

\begin{proof}
By Remark \ref{Rem: W_bound_AVD_SC}, we have $t\dot{W} + (\alpha-3)W \le 0$, which implies that
$$ t^{\alpha-3}W(t) \le T^{\alpha-3}W(T),\qbox{and} W(t) = \mathcal{O}\left( \frac{1}{t^{\alpha-3}} \right). $$

(i) By strong convexity of $f$ and $g$, we have
\begin{align*}
\langle \nabla f(x), x-x^* \rangle - ( f(x) - f(x^*) ) &\ge \frac{\mu_f}{2}\| x-x^* \|^2,\\
\langle \nabla g(y), y-y^* \rangle - ( g(y) - g(y^*) ) &\ge \frac{\mu_g}{2}\| y - y^* \|^2,
\end{align*}
which, combined with \eqref{E: W_AVD}, gives
$$ \frac{(\alpha-1)t}{2}\min(\mu_f\beta_f,\mu_g\beta_g) \| s-s^* \|^2 \le W(t). $$
As a result, 
$$ \| s-s^* \| = \mathcal{O}\left( \frac{1}{t^{(\alpha-2)/2}} \right). $$

(ii) In view of \eqref{E: W_AVD}, we have 
$$ \| t\dot{s} + (\alpha-1)(s-s^*) \| = \mathcal{O}\left( \frac{1}{t^{(\alpha-3)/2}} \right). $$
Using the triangle inequality, we have
$$ t\| \dot{s} \| \le \| t\dot{s} + (\alpha-1)(s-s^*) \| + (\alpha-1)\| s-s^* \| = \mathcal{O}\left( \frac{1}{t^{(\alpha-3)/2}} \right), $$
which implies that
$$ \| \dot{s} \| = \mathcal{O}\left( \frac{1}{t^{(\alpha-1)/2}} \right). $$ 

(iii) Likewise, it follows from \eqref{E: W_AVD} that
$$ \mathcal{L}(x(t),y^*) - \mathcal{L}(x^*,y(t)) \le \frac{W(t)}{t(\gamma t + r)} = \mathcal{O}\left( \frac{1}{t^{\alpha-1}} \right), $$
which allows us to conclude.
\end{proof}

\begin{remark} 
This result is enabled by inclusion of the Hessian-driven damping ($\beta_f,\beta_g>0$), which demonstrates the advantages of including the Hessian-driven damping in the dynamics.
\end{remark}

\section{Inertial Dynamics II: The Strongly Convex Case}\label{Sec: PD-HBF-H}
In this section, we consider problem \eqref{Prob: P} in case $f$ and $g$ are strongly convex and their strong convexity parameters are known. We propose to solve it using the following inertial dynamics with Hessian-driven damping:
\begin{equation}\label{System: PD_system_HBF}\tag{PD-HBF-H}
\left\{
\begin{array}{lcl}
\ddot{x} + \alpha\dot{x} + \beta_f\nabla_{xx}\mathcal{L}(x,y)\dot{x} + \gamma \nabla_x \mathcal{L}(x,y+\theta\dot{y}) &=& 0,\\[4pt]
\ddot{y} + \alpha\dot{y} - \beta_g\nabla_{yy}\mathcal{L}(x,y)\dot{y} - \gamma \nabla_y \mathcal{L}(x+\theta\dot{x},y) &=& 0,
\end{array} 
\right.
\end{equation}
where $\alpha,\beta_f,\beta_g,\gamma,\theta >0$, with $t\ge t_0$ for some $t_0\ge 0$. Notice that \eqref{System: PD_system_HBF} is different from \eqref{E: bilinear_HBF} \cite{He_2025_PD_SC}, in its inclusion of the Hessian-driven damping ($\beta_f,\beta_g>0$). 

As in Subsection \ref{Subsec: F}, we define $s=(x,y)$, and introduce $F(s)$ (see \eqref{E: F} ) so as to simplify the notation. By doing so, \eqref{System: PD_system_HBF} can be rewritten as
\begin{equation}\label{E: PD_system_HBF}
\ddot{s} + \alpha \dot{s} + \begin{bmatrix}
\beta_f \nabla^2 f(x) \dot{x}\\[2pt]
\beta_g \nabla^2 g(y) \dot{y}
\end{bmatrix} + \gamma \nabla F + \gamma \begin{bmatrix}
A^* ( \theta \dot{y} + y - y^* ) \\[2pt]
- A( \theta\dot{x} + x - x^* )
\end{bmatrix} = 0.
\end{equation}

\subsection{Energy Estimations}
Our analysis centers around the energy function $W:[t_0,\infty)\to\R$, defined by
\begin{align}\label{E: W_HBF}
W(t) 
&= \frac{1}{2}\| \dot{s} + \xi(s-s^*) \|^2 + \gamma F(s) 
 + \beta_f \xi \left[ \langle \nabla f(x), x-x^* \rangle - \left( f(x) - f(x^*) \right) \right] \nonumber\\
&\quad + \beta_g \xi \left[ \langle \nabla g(y), y-y^* \rangle - \left( g(y) - g(y^*) \right) \right],
\end{align}
where $0\le\xi\le\alpha$. Using \eqref{E: W_HBF}, we have the following:

\begin{lemma}\label{Lem: W_bound_HBF}
Let $W$ be defined by \eqref{E: W_HBF}. Set $\xi = \frac{1}{\theta}$. Then, we have
\begin{align*}
\dot{W}(t) + (\alpha-\xi)W(t)
&= \frac{\xi^2(\alpha-\xi)}{2}\| s-s^* \|^2 
  - \frac{\alpha-\xi}{2}\| \dot{s} \|^2
  - \gamma(2\xi-\alpha) F \\
&\quad - \xi\left( \gamma-(\alpha-\xi)\beta_f \right) [ \langle \nabla f(x), x-x^* \rangle - \left( f(x) - f(x^*) \right) ] \\
&\quad - \xi\left( \gamma-(\alpha-\xi)\beta_g \right) [ \langle \nabla g(y), y-y^* \rangle - \left( g(y) - g(y^*) \right) ] \\
&\quad -\beta_f \langle \nabla^2 f(x)\dot{x}, \dot{x} \rangle 
  -\beta_g \langle \nabla^2 g(y)\dot{y}, \dot{y} \rangle.
\end{align*}
\end{lemma}

\begin{proof}
Denote $W(t) := \sum_{i=1}^4 W_i(t)$, with $W_i$ $(i=1,2,3,4)$ being the terms in \eqref{E: W_HBF} sequentially. We begin by using \eqref{E: PD_system_HBF}, to obtain
\begin{align*}
\dot{W}_1(t)
&= -\langle \dot{s} + \xi(s-s^*), (\alpha-\xi)\dot{s} \rangle
   - \left\langle \dot{s} + \xi(s-s^*), \begin{bmatrix}
\beta_f \nabla^2 f(x) \dot{x}\\[2pt]
\beta_g \nabla^2 g(y) \dot{y}
\end{bmatrix} \right\rangle \\
&\quad - \gamma \langle \nabla F, \dot{s} + \xi(s-s^*) \rangle 
 - \gamma \left\langle \dot{s} + \xi(s-s^*), \begin{bmatrix}
A^* ( \theta \dot{y} + y - y^* ) \\[2pt]
- A( \theta\dot{x} + x - x^* )
\end{bmatrix} \right\rangle.
\end{align*}
Since $\xi = \frac{1}{\theta}$, we have
$$ \left\langle \dot{s} + \xi(s-s^*), \begin{bmatrix}
A^* ( \theta \dot{y} + y - y^* ) \\[2pt]
- A( \theta\dot{x} + x - x^* )
\end{bmatrix} \right\rangle 
= \frac{1}{\theta}\left\langle \begin{bmatrix}
\theta\dot{x} + x-x^*\\[2pt]
\theta\dot{y} + y-y^*
\end{bmatrix}, \begin{bmatrix}
A^* ( \theta \dot{y} + y - y^* ) \\[2pt]
- A( \theta\dot{x} + x - x^* )
\end{bmatrix} \right\rangle = 0. $$
As a result,
\begin{align*}
\dot{W}_1 (t)
&= -(\alpha-\xi) \| \dot{s} \|^2 - \xi(\alpha-\xi)\langle s-s^*, \dot{s} \rangle
   -\beta_f \langle \nabla^2 f(x)\dot{x}, \dot{x} \rangle 
   -\beta_g \langle \nabla^2 g(y)\dot{y}, \dot{y} \rangle\\
&\quad - \beta_f \xi \langle \nabla^2 f(x)\dot{x}, x-x^* \rangle
   - \beta_g \xi \langle \nabla^2 g(y)\dot{y}, y-y^* \rangle 
   -\gamma \langle \nabla F, \dot{s} \rangle 
   -\gamma\xi \langle \nabla F, s-s^* \rangle.
\end{align*}
Likewise, we obtain
\begin{align*}
\dot{W}_2(t) &= \gamma \langle \nabla F, \dot{s} \rangle,\\
\dot{W}_3(t) &= \beta_f \xi \langle \nabla^2 f(x)\dot{x}, x-x^* \rangle,\\
\dot{W}_4(t) &= \beta_g \xi \langle \nabla^2 g(y)\dot{y}, y-y^* \rangle.
\end{align*}
Summing up these equations, it follows that
\begin{align*}
\dot{W}(t) 
&= -(\alpha-\xi) \| \dot{s} \|^2 - \xi(\alpha-\xi)\langle s-s^*, \dot{s} \rangle
  -\gamma\xi \langle \nabla F, s-s^* \rangle
  -\beta_f \langle \nabla^2 f(x)\dot{x}, \dot{x} \rangle \\
& -\beta_g \langle \nabla^2 g(y)\dot{y}, \dot{y} \rangle,
\end{align*}
which, combined with Lemma \ref{Lem: dF}, gives
\begin{align*}
\dot{W}(t)
&= -\frac{\alpha-\xi}{2}\left\| \dot{s} + \xi(s-s^*) \right\|^2
  + \frac{\xi^2(\alpha-\xi)}{2}\| s-s^* \|^2 
  - \frac{\alpha-\xi}{2}\| \dot{s} \|^2 
  - \gamma\xi F \\
&\quad -\gamma\xi [ \langle \nabla f(x), x-x^* \rangle - \left( f(x) - f(x^*) \right) ] 
 -\gamma\xi [ \langle \nabla g(y), y-y^* \rangle - \left( g(y) - g(y^*) \right) ]\\
&\quad -\beta_f \langle \nabla^2 f(x)\dot{x}, \dot{x} \rangle 
  -\beta_g \langle \nabla^2 g(y)\dot{y}, \dot{y} \rangle.   
\end{align*}
Upon adding this equation to
\begin{align*}
(\alpha-\xi)W(t)
&= \frac{\alpha-\xi}{2}\| \dot{s} + \xi(s-s^*) \|^2 + \gamma(\alpha-\xi) F \\
&\quad + \beta_f \xi (\alpha-\xi) [ \langle \nabla f(x), x-x^* \rangle - \left( f(x) - f(x^*) \right) ] \\
&\quad + \beta_g \xi (\alpha-\xi) [ \langle \nabla g(y), y-y^* \rangle - \left( g(y) - g(y^*) \right) ],   
\end{align*}
the result follows immediately.
\end{proof}

\subsection{Convergence Analysis}

\begin{theorem}
Let $f:\X\to\R$ and $g:\Y\to\R$ be strongly convex with parameters $\mu_f,\mu_g>0$ and continuously differentiable. Let $(x,y):[t_0,\infty)\to\X\times\Y$ be a solution of the system \eqref{System: PD_system_HBF}, where
\begin{align*}
& \alpha\le 2\sqrt{ \min(\mu_f,\mu_g) \gamma_0 },\qbox{} \min(\beta_f,\beta_g) > 0,\qbox{} \gamma = \gamma_0 + \frac{\alpha}{2} \max(\beta_f,\beta_g),\\
&\qbox{with} \theta = \frac{2}{\alpha},\qbox{and}\gamma_0 >0. 
\end{align*}
Then for every $t\ge t_0$, we have
\begin{enumerate}[label=(\roman*), itemsep=0.7em, leftmargin=12mm]
\item ${\| x(t) - x^* \|^2 + \| y(t) - y^* \|^2} = \mathcal{O}\left( e^{-\frac{1}{2}\alpha t} \right)$.

\item ${ \| \dot{x}(t) \|^2 + \| \dot{y}(t) \|^2 } = \mathcal{O}\left( e^{-\frac{1}{2}\alpha t} \right)$.

\item $\mathcal{L}(x(t),y^*) - \mathcal{L}(x^*,y(t)) = \mathcal{O}\left( e^{-\frac{1}{2}\alpha t} \right)$. 
\end{enumerate}
\end{theorem}

\begin{proof}
Since $f$ and $g$ are strongly convex, we have
$$\langle \nabla^2 f(x)\dot{x}, \dot{x} \rangle \ge 0, \quad \langle \nabla^2 g(y)\dot{y}, \dot{y} \rangle\ge 0, $$
and
\begin{align}
\label{E: PD_HBF_SC_bound_1}
\langle \nabla f(x), x-x^* \rangle - \left( f(x) - f(x^*) \right) 
&\ge \frac{\mu_f}{2}\| x-x^* \|^2,\\
\label{E: PD_HBF_SC_bound_2}
\langle \nabla g(y), y-y^* \rangle - \left( g(y) - g(y^*) \right) 
&\ge \frac{\mu_g}{2}\| y-y^* \|^2.
\end{align}
Setting $\xi = \frac{\alpha}{2}$ in Lemma \ref{Lem: W_bound_HBF}, we obtain
\begin{align*}
\dot{W} + \frac{\alpha}{2}W 
&\le -\tfrac{\alpha}{4}\left[ \mu_f\left( \gamma - \tfrac{1}{2}\alpha\beta_f \right) - \tfrac{\alpha^2}{4} \right] \| x-x^* \|^2 
 -\tfrac{\alpha}{4}\left[ \mu_g\left( \gamma - \tfrac{1}{2}\alpha\beta_g \right) - \tfrac{\alpha^2}{4} \right] \| y-y^* \|^2.
\end{align*}
Keeping in mind that
$$ \alpha\le 2\sqrt{ \min(\mu_f,\mu_g) \gamma_0 },\qbox{and} \gamma = \gamma_0 + \frac{\alpha}{2} \max(\beta_f,\beta_g), $$
we deduce $\dot{W} + \frac{\alpha}{2}W \le 0$, which gives 
$$ W(t) \le W(t_0) e^{-\frac{1}{2}\alpha t}. $$

(i) Using \eqref{E: PD_HBF_SC_bound_1} and \eqref{E: PD_HBF_SC_bound_2} in \eqref{E: W_HBF} gives
$$ \frac{\xi}{2}\min(\mu_f\beta_f,\mu_g\beta_g) \| s-s^* \|^2 \le W(t) \le W(t_0) e^{-\frac{1}{2}\alpha t}, $$
and the result follows.

(ii) Likewise, we have 
$$ \| \dot{s} + \xi(s-s^*) \| \le \sqrt{2W(t)} \le \sqrt{2W(t_0)} e^{-\frac{1}{4}\alpha t }. $$
Using the triangle inequality, we have
$$ \| \dot{s} \| \le \| \dot{s} + \xi(s-s^*) \| + \xi \| s-s^* \| \le \left( \sqrt{ 2W(t_0)} + \sqrt{\tfrac{2\xi W(t_0)}{\min(\mu_f\beta_f,\mu_g\beta_g)} } \right) e^{-\frac{1}{4}\alpha t } . $$

(iii) The argument follows by applying $ F(s) \le \frac{1}{\gamma}W(t) $.

(iv) In view of optimality conditions and
\begin{align*}
\left[ \langle \nabla f(x), x-x^* \rangle - \left( f(x) - f(x^*) \right) \right] 
&\le W(t)/(\beta_f\xi),\\
\left[ \langle \nabla g(y), y-y^* \rangle - \left( g(y) - g(y^*) \right) \right]
&\le W(t)/(\beta_g\xi),
\end{align*}
the result follows by applying smoothness.
\end{proof}

\begin{remark}
Setting $\gamma_0=1$, $\alpha = 2\sqrt{\min(\mu_f,\mu_g)}$ and $\gamma = 1 + \frac{\alpha}{2} \max(\beta_f,\beta_g)$, one obtains an accelerated linear convergence rate:
$$ \mathcal{L}(x(t),y^*) - \mathcal{L}(x^*,y(t)) = \mathcal{O}\left( e^{-\sqrt{\min(\mu_f,\mu_g)} t} \right), $$
for the inertial primal dual dynamical system \eqref{System: PD_system_HBF}, matching the one derived for a heavy-ball system \eqref{E: bilinear_HBF} in \cite{He_2025_PD_SC}. While preserving fast convergence properties, smoother trajectories can be expected for \eqref{System: PD_system_HBF}, due to its inclusion of the Hessian-driven damping $\beta_f\nabla_{xx}\mathcal{L}(x,y)\dot{x}$ and $\beta_g\nabla_{yy}\mathcal{L}(x,y)\dot{y}$.
\end{remark}

\begin{remark} \label{R:convergent_gradient_bis}
    As discussed in Remark \ref{R:convergent_gradient}, if $f$ is (locally) smooth, then $\|\nabla f(x(t))+A^*y^*\|^2=\mathcal{O}\left( e^{-\frac{1}{2}\alpha t} \right)$. Similarly, if $g$ is (locally) smooth, then $\|\nabla g(y(t))-Ax^*\|^2=\mathcal{O}\left( e^{-\frac{1}{2}\alpha t} \right)$.
\end{remark}

\section{Inertial Dynamics III: The Affinely Constrained Case}\label{Sec: affine convex}
In this section, we consider problem \eqref{Prob: min_fx}, which is \eqref{Prob: P} with $g=\langle b, y\rangle$.

\subsection{Augmented Lagrangian and Inertial Primal Dual Dynamics}\label{Sec: ALM}
The optimality conditions for \eqref{Prob: min_fx} are given by
\begin{equation}\label{E: opt_cond}
\begin{aligned}
\nabla f(x) + A^* y = 0,\quad
Ax = b,
\end{aligned}
\end{equation} 
where $A^*:\Y\to\X$ is the adjoint operator of $A$, and $y\in\Y$ is the Lagrange multiplier. In what follows, we assume that the solution set $\mathcal{S}$ to the above equations is nonempty.

The augmented Lagrangian $\mathcal{L}_\rho:\X\times\Y\to\R$ is defined by \eqref{E: Lm}:
$$ \mathcal{L}_\rho(x,y) = f(x) + \langle y, Ax-b \rangle + \frac{\rho}{2}\| Ax-b \|^2, $$
where $\rho>0$. If $\rho=0$, $\mathcal{L}_\rho(x,y)$ reduces to the regular Lagrangian $\mathcal{L}(x,y)$. A point $(x,y)$ satisfies \eqref{E: opt_cond} if, and only if, $\nabla\mathcal{L}_\rho(x,y)=0$. Therefore, \eqref{Prob: min_fx} is equivalent to a saddle-point problem of finding $(x^*,y^*)\in\X\times\Y$ such that
$$ \mathcal{L}_\rho(x^*,y) \le \mathcal{L}_\rho(x^*,y^*) \le \mathcal{L}_\rho(x,y^*),\quad \forall\, (x,y)\in\X\times\Y.$$

To solve problem \eqref{Prob: min_fx}, we propose to follow the trajectories of the inertial primal dual dynamical system with Hessian-driven damping:  
\begin{equation}\label{inertial_primal_dual_system}\tag{PD-AVD-H}
\left\{
\begin{array}{rcl}
\ddot{x} + \frac{\alpha}{t}\dot{x} + \beta\nabla^2_{xx}\mathcal{L}_\rho \left( x, y \right) \dot{x} 
 + \left( \gamma + \frac{r}{t} \right) \nabla_x \mathcal{L}_\rho\left( x, y + \theta t \dot{y} \right) &=& 0, \\[2pt]
\ddot{y} + \frac{\alpha}{t}\dot{y} - \left( \gamma + \frac{r}{t} \right) \nabla_y \mathcal{L}_\rho\left( x + \theta t \dot{x}, y \right) &=& 0,
\end{array}
\right.
\end{equation}
where $\alpha,\beta,\gamma,\theta,r > 0$. The initial conditions are given by $\left( x(t_0), y(t_0) \right) = \left( x_0,y_0 \right)$ and $\left( \dot{x}(t_0),\dot{y}(t_0) \right) = \left( \dot{x}_0, \dot{y}_0 \right)$.

Notice that system \eqref{inertial_primal_dual_system} is different from the one proposed in \cite{He_2025}, where the system is structured as second order primal dynamics and first order dual dynamics. Notice also that system \eqref{inertial_primal_dual_system} is distinct from the one in \cite{Radu_2025_OGDA}, which has a time-dependent Hessian-driven damping. In the unconstrained case, the system in \cite{Radu_2025_OGDA} cannot recover \eqref{E: AVD-H} while preserving the convergence rate $\mathcal{O}\left( \frac{1}{t^2} \right)$.

To simplify the notation, we define the augmented state vector $s := (x, y)$ and the difference of the augmented Lagrangians
\begin{equation}\label{E: F_ALM}
F(x) := \mathcal{L}_\rho(x,y^*) - \mathcal{L}_\rho(x^*,y) 
= f(x) - f(x^*) + \langle y^*, Ax-b \rangle + \frac{\rho}{2}\| Ax-b \|^2.
\end{equation}
Notice that $F(x)\ge 0$ for all $x\in\X$, and
\begin{equation*}
\nabla F(x) = 
\nabla f(x) + A^* y^* + \rho A^*(Ax-b).
\end{equation*}  
As a result, we can rewrite \eqref{inertial_primal_dual_system} as
\begin{equation}\label{E: primal_dual}
\begin{aligned}
&&\ddot{s} + \frac{\alpha}{t}\dot{s}+ \beta\begin{bmatrix}
 \left( \nabla^2 f(x) + \rho A^* A \right) \dot{x}  \\
0
\end{bmatrix}  
&& + \left( \gamma + \frac{r}{t} \right) \begin{bmatrix}
\nabla F(x)+A^*\big( \theta t \dot{y} + y - y^* \big)\\
-A \big( \theta t \dot{x} + x-x^* \big)
\end{bmatrix} = 0,
\end{aligned}
\end{equation}
which will be used for the subsequent convergence analysis.

\subsection{Energy Estimations and Trajectory Boundedness}
Consider the energy function $W:[t_0,\infty)\to\R$, defined by
\begin{equation}\label{E: W}
\begin{aligned}
W &= \frac{1}{2}\| t\dot{s} + \xi(s-s^*) \|^2 + \frac{\eta}{2}\| s-s^* \|^2 + t(\gamma t + r) F(x) \\ 
&\quad + \beta\xi t \left[ \langle \nabla f(x), x-x^* \rangle - \left( f(x) - f(x^*) \right) \right] 
 + \frac{\beta\xi \rho t}{2}\| Ax-b \|^2,
\end{aligned}
\end{equation}
where $0\le\xi\le \alpha-1$ and $\eta = \xi(\alpha-1-\xi)\ge 0$. Observe that $W(t)\ge 0$ for all $t\ge t_0$, in view of the definition of $F$ and the convexity of $f$. We have the following:

\begin{lemma}\label{Lem: epsilon_dev}
Let $W$ be defined by \eqref{E: W}. Set $\xi = \frac{1}{\theta}$. For every $t\ge t_0$, we have
\begin{align*}
\dot{W}
&= -(\alpha-\xi-1)t\| \dot{s} \|^2
   -\beta t^2 \left\langle \left( \nabla^2 f(x) + \rho A^* A \right)\dot{x}, \dot{x} \right\rangle \\
&\quad - \frac{\rho\xi}{2}( \gamma t + r - \beta ) \| Ax-b \|^2   
  - [ (\xi-2)\gamma t + (\xi-1)r ] F(x) \\
&\quad -\xi( \gamma t + r - \beta ) \left[ \langle \nabla f(x), x-x^* \rangle - \left( f(x) - f(x^*) \right) \right].
\end{align*}
\end{lemma}

\begin{proof}
Denote $W := \sum_{i=1}^5 W_i$, with $W_i$ $(i=1,2,\cdots,5)$ being the terms in \eqref{E: W} sequentially. We begin by using \eqref{E: primal_dual}, to obtain
\begin{align*}
\dot{W}_1
&= \langle t\dot{s} + \xi(s-s^*), t\ddot{s} + (\xi + 1)\dot{s} \rangle \\
&= \langle t\dot{s} + \xi(s-s^*), -(\alpha-\xi-1)\dot{s} \rangle  
  - (\gamma t + r) \left\langle t\dot{x} + \xi(x-x^*), \nabla F(x) \right\rangle \\
&\quad - (\gamma t + r) \left\langle t\dot{s} + \xi(s-s^*), \begin{bmatrix}
 A^*\left( \theta t \dot{y} + (y - y^*) \right)\\
 -A \left( \theta t \dot{x} + (x-x^*) \right)
 \end{bmatrix} \right\rangle \\
&\quad - \beta t \left\langle t\dot{x} + \xi(x-x^*), \left( \nabla^2 f(x) + \rho A^* A \right)\dot{x} \right\rangle.
\end{align*}
Since $\xi = \frac{1}{\theta}$, we have
\begin{align*}
&\quad \left\langle t\dot{s} + \xi(s-s^*), \begin{bmatrix}
A^*\big( \theta t \dot{y} + (y - y^*) \big)\\
 -A \left( \theta t \dot{x} + (x-x^*) \right)
 \end{bmatrix} \right\rangle \\ 
 &= \frac{1}{\theta}\left\langle \begin{bmatrix}
 \theta t \dot{x} + x - x^*\\
\theta t \dot{y} + y-y^* 
 \end{bmatrix}, \begin{bmatrix}
 A^*\big( \theta t \dot{y} + y - y^* \big)\\
 -A \left( \theta t \dot{x} + x-x^* \right)
 \end{bmatrix} \right\rangle=0.
\end{align*}
As a result, 
\begin{align*}
\dot{W}_1
&= -(\alpha-\xi-1)t \| \dot{s} \|^2
   -\xi(\alpha-\xi-1) \langle s-s^*, \dot{s} \rangle \\ 
&\quad -t(\gamma t + r) \langle \nabla F(x), \dot{x} \rangle  
 -(\gamma t + r)\xi \langle \nabla F(x), x-x^* \rangle \\
&\quad -\beta t^2 \left\langle \left( \nabla^2 f(x) + \rho A^* A \right)\dot{x}, \dot{x} \right\rangle 
 - \beta\xi t \left\langle \nabla^2 f(x)\dot{x}, x-x^* \right\rangle  
 - \beta\xi \rho t \left\langle A^* A\dot{x}, x-x^* \right\rangle.
\end{align*}
Likewise, we obtain
\begin{align*}
\dot{W}_2
&= \eta \langle s-s^*, \dot{s} \rangle, \\
\dot{W}_3
&= t(\gamma t + r) \langle \nabla F(x), \dot{x} \rangle
  + (2\gamma t + r) F(x), \\
\dot{W}_4 
&= \beta\xi t \left\langle \nabla^2 f(x) \dot{x}, x-x^* \right\rangle  
  + \beta\xi \left[ \langle \nabla f(x), x-x^* \rangle - \left( f(x) - f(x^*) \right) \right], \\
\dot{W}_5 
&= \beta\xi \rho t \langle Ax-b, A\dot{x} \rangle + \frac{\beta\xi \rho}{2}\| Ax-b \|^2 
= \beta\xi \rho t \langle x-x^*, A^*A\dot{x} \rangle 
  + \frac{\beta\xi \rho}{2}\| Ax-b \|^2.
\end{align*}
Summing up these equations, and recalling that $\eta = \xi(\alpha-\xi-1)$, it follows that
\begin{equation}\label{E: W_dev_temp}
\begin{aligned}
\dot{W} 
&= -(\alpha-\xi-1)t \| \dot{s} \|^2
   -\beta t^2 \left\langle \left( \nabla^2 f(x) + \rho A^* A \right)\dot{x}, \dot{x} \right\rangle  
   + \frac{\beta\xi \rho}{2}\| Ax-b \|^2 \\
&\quad -(\gamma t + r)\xi \langle \nabla F(x), x-x^* \rangle 
   + (2\gamma t + r) F(x) 
   + \beta\xi \left[ \langle \nabla f(x), x-x^* \rangle - \left( f(x) - f(x^*) \right) \right].
\end{aligned}
\end{equation}
With $Ax^* = b$ and \eqref{E: F_ALM} in mind, we have
\begin{align*}
\langle \nabla F(x), x-x^* \rangle 
&= \left\langle \nabla f(x) + A^* y^* + \rho A^*(Ax-b), x-x^* \right\rangle \\
&= \langle \nabla f(x), x-x^* \rangle + \langle y^*, Ax-b \rangle + \rho\| Ax-b\|^2 \\
&= \left[ \langle \nabla f(x), x-x^* \rangle - \left( f(x) - f(x^*) \right) \right] 
  + F(x) + \frac{\rho}{2}\| Ax-b\|^2.
\end{align*}
We conclude by substituting this into \eqref{E: W_dev_temp}.
\end{proof}

\begin{remark}\label{Rem: epsilon_dev_bound}
Since $f$ is convex, we have 
$$ \left\langle \big( \nabla^2 f(x) + \rho A^* A \big)\dot{x}, \dot{x} \right\rangle 
= \left\langle \nabla^2 f(x) \dot{x}, \dot{x} \right\rangle + \rho \| A\dot{x} \|^2 \ge 0.$$
If $\alpha\ge 3$, $r\ge\beta$ and $\frac{1}{\alpha-1}\le \theta \le \frac{1}{2}$, then $W(t)$ is nonincreasing, whence $W(t)\le W(t_0)$ for all $t\ge t_0$.
\end{remark}

Using Lemma \ref{Lem: epsilon_dev}, we can establish the boundedness of the trajectories of \eqref{inertial_primal_dual_system}, and the $\mathcal{O}(\frac{1}{t})$ decay of the velocity, as in Subsection \ref{Subsec: convergence_analysis_PD_AVD-H}. To avoid repetition, we only state the result here and leave the proof to the reader.

\begin{theorem}
Let $(x,y):[t_0,\infty)\to\X\times\Y$ be a solution of the system \eqref{inertial_primal_dual_system}, where
$$\alpha\ge 3,\quad r\ge \beta,\quad\text{and}\quad \frac{1}{\alpha-1}\le\theta\le \frac{1}{2}.$$
Then, the following hold:
\begin{enumerate}[label=(\roman*), itemsep=0.7em, leftmargin=12mm]
\item The trajectory is bounded.

\item The velocity satisfies $\| \dot{x}(t) \|^2 + \| \dot{y}(t) \|^2   = \mathcal{O}\left( \frac{1}{t^2} \right).$ 
\end{enumerate}
\end{theorem}

\begin{remark}
Notice that the trajectory boundedness and the $\mathcal{O}\left( \frac{1}{t} \right)$ decay of the velocity still hold when $\alpha=3$, which complements the result in \cite{Zeng_2023,Radu_2021}.
\end{remark}

\begin{remark}
    As was discussed in Remarks \ref{R:convergent_gradient} and \ref{R:convergent_gradient_bis}, if $f$ is (locally) smooth, then $\|\nabla f(x) + A^* y^* + \rho A^*(Ax^*-b)\|^2=\mathcal{O}(\frac{1}{t})$.
\end{remark}

\subsection{Convergence Rate Analysis}
In this subsection, we prove the convergence rates for the primal dual gap, the feasibility measure and the function values of the system \eqref{inertial_primal_dual_system}.

Given $\ell\in\Y$, define the energy function $E_\ell:[t_0,\infty)\to\R$, by
\begin{equation}\label{E: epsilon_bar}
\begin{aligned}
E_\ell
&= \frac{1}{2}\| t\dot{s}(t) + \xi(s(t)-s_\ell^*) \|^2 + \frac{\eta}{2}\| s(t)-s_\ell^* \|^2
  + t(\gamma t + r) \big[ \mathcal{L}_\rho\left(x(t),\ell\right) - \mathcal{L}_\rho(x^*,\ell)\big] \\
&\quad + \beta\xi t \left[ \langle \nabla f(x(t)), x(t)-x^* \rangle - \left( f(x(t)) - f(x^*) \right) \right] 
  + \frac{\beta\xi \rho t}{2}\| Ax(t)-b \|^2,
\end{aligned}
\end{equation}
where $s_\ell^*=(x^*,\ell)$. Clearly, we have $E_{y^*}=W$.

\begin{remark}
The term $\mathcal{L}_\rho\left(x,\ell\right) - \mathcal{L}_\rho(x^*,\ell)$ in $E_\ell$ is not necessarily nonnegative, since 
\begin{equation}\label{E: Lagrangian_bar}
\mathcal{L}_\rho\left(x,\ell\right) - \mathcal{L}_\rho(x^*,\ell) = F(x) 
  + \left\langle \ell - y^*, Ax-b \right\rangle,
\end{equation}
and $\left\langle \ell - y^*, Ax-b \right\rangle$ can be negative. The rest of the terms in $E_\ell$ are clearly nonnegative.
\end{remark}

Before we use $E_\ell$ for the convergence proof, we derive some estimations, which will facilitate the analysis.

\begin{lemma}\label{Lem: epsilon_bar}
Let $(x,y):[t_0,\infty)\to\X\times\Y$ be a solution of \eqref{inertial_primal_dual_system}. For $\ell\in\Y$, let $E_\ell:[t_0,\infty)\to\R$ be defined by \eqref{E: epsilon_bar}. Then, we have
\begin{equation*}
E_\ell \le 2W + (\xi^2 + \eta)\| \ell - y^*\|^2 + t(\gamma t + r)\left\langle \ell - y^*, Ax(t)-b \right\rangle,
\end{equation*}
for every $t\ge t_0$. In particular,
$$\sup_{\|\ell-y^*\|\le 1} E_\ell(t_0) \le C_0:=2W(t_0) + \xi^2 + \eta + t_0(\gamma t_0 + r) \| Ax_0-b \|.$$
\end{lemma}

\begin{proof}
From \eqref{E: epsilon_bar}, we get
$$\frac{1}{2}\| t\dot{s} + \xi(s-s_\ell^*) \|^2
\le \| t\dot{s} + \xi(s-s^*) \|^2 + \xi^2\| \ell - y^*\|^2,$$
and
$$ \frac{\eta}{2}\| s-s_\ell^* \|^2 
\le \eta\| s-s^* \|^2 + \eta \left\| \ell - y^* \right\|^2. $$
We conclude using \eqref{E: Lagrangian_bar}.
\end{proof}

We are now in a position to prove the main convergence results of this section.

\begin{theorem} \label{T:rates}
Let $(x,y):[t_0,\infty)\to\X\times\Y$ be a solution of the system \eqref{inertial_primal_dual_system}, where
$$ \alpha\ge 3,\quad r\ge\beta \qbox{and} \frac{1}{\alpha-1}\le \theta \le \frac{1}{2}.$$
Let $C_0$ be the constant defined in Lemma \ref{Lem: epsilon_bar}. For every $t\ge t_0$, we have
\begin{equation} \label{E:main_thm}
    0\le \mathcal{L}_\rho\left(x(t),y^*\right) - \mathcal{L}_\rho(x^*,y^*) + \|Ax(t)-b\| \le \frac{2C_0}{t(\gamma t + r)},
\end{equation}
and
$$ -\frac{2C_0 \| y^* \| }{t(\gamma t + r)} \le f\left( x(t) \right) - f(x^*) \le \frac{2C_0 \left(1+\|y^* \|\right)}{t(\gamma t + r)}. $$
\end{theorem}

\begin{proof}
Fix $\tau\ge t_0$, and set
\begin{equation}\label{E: lambda_tau}
\ell_\tau = \left\{
\begin{array}{ccl}
y^* + \frac{Ax(\tau)-b}{\| Ax(\tau)-b \|},& & \text{if }Ax(\tau)-b\neq 0,\\
y^* & & \text{if }Ax(\tau)-b= 0.
\end{array}
\right.
\end{equation}
Since $\|\ell_\tau-y^*\|\le 1$, it follows from Lemma \ref{Lem: epsilon_bar} that 
$$ E_{\ell_\tau}( t_0) \le C_0,$$
and
\begin{equation}\label{E: E_tau_bound}
\begin{aligned}
E_{\ell_\tau}(t) & \le 2W(t) + \xi^2 + \eta + t(\gamma t + r)\left\langle \ell_\tau - y^*, Ax(t)-b \right\rangle \\
 & \le 2W(t_0) + \xi^2 + \eta + t(\gamma t + r)\left\langle \ell_\tau - y^*, Ax(t)-b \right\rangle,    
\end{aligned}
\end{equation}
in view of Remark \ref{Rem: epsilon_dev_bound}. On the other hand, setting $\xi = \frac{1}{\theta}$, and proceeding as in the proof of Lemma \ref{Lem: epsilon_dev}, we obtain
\begin{align*}
\dot E_{\ell_\tau}(t)  
&\le -[ (\xi-2)\gamma t + (\xi-1)r ] \left[ \mathcal{L}_\rho\left(x(t),\ell_\tau\right) - \mathcal{L}_\rho(x^*,\ell_\tau) \right] \\
&\le -[ (\xi-2)\gamma t + (\xi-1)r ] \left\langle \ell_\tau - y^*, Ax(t)-b \right\rangle,
\end{align*}
where the last inequality is due to \eqref{E: Lagrangian_bar}. Multiplying both sides of this inequality by $\frac{t^{\xi-1}}{\gamma t+r}$ gives
\begin{align*}
\left(\tfrac{t^{\xi-1}}{\gamma t+r}\right)\dot E_{\ell_\tau}(t) 
& \le -\tfrac{t^{\xi-1}}{\gamma t+r}[ (\xi-2)\gamma t + (\xi-1)r ] \left\langle \ell_\tau - y^*, Ax(t)-b \right\rangle \\
& = -t(\gamma t+r) \left\langle \ell_\tau - y^*, Ax(t)-b \right\rangle \frac{d}{dt}\left(\tfrac{t^{\xi-1}}{\gamma t+r}\right).
\end{align*}
It follows that
\begin{align*}
\frac{d}{dt}\left[\left(\tfrac{t^{\xi-1}}{\gamma t+r}\right)E_{\ell_\tau}(t)\right] 
&= E_{\ell_\tau}(t)\frac{d}{dt}\left(\tfrac{t^{\xi-1}}{\gamma t+r}\right) + \left(\tfrac{t^{\xi-1}}{\gamma t+r}\right)\dot E_{\ell_\tau}(t)  \\
&\le \left[E_{\ell_\tau}(t)-t(\gamma t+r) \left\langle \ell_\tau - y^*, Ax(t)-b \right\rangle\right]\frac{d}{dt}\left(\tfrac{t^{\xi-1}}{\gamma t+r}\right) \\
&\le \left[2W(t_0) + \xi^2 + \eta \right] \frac{d}{dt}\left(\tfrac{t^{\xi-1}}{\gamma t+r}\right),
\end{align*}
by applying \eqref{E: E_tau_bound}. Integrating from $t_0$ to $t$, we obtain
\begin{align*}
\left(\tfrac{t^{\xi-1}}{\gamma t+r}\right)E_{\ell_\tau}(t)-\left(\tfrac{t_0^{\xi-1}}{\gamma t_0+r}\right)E_{\ell_\tau}(t_0) 
&\le  \left[2W(t_0) + \xi^2 + \eta \right] \left(\tfrac{t^{\xi-1}}{\gamma t+r}-\tfrac{t_0^{\xi-1}}{\gamma t_0+r}\right),
\end{align*}
which implies that
$$E_{\ell_\tau}(t) \le E_{\ell_\tau}(t_0) + \left[2W(t_0) + \xi^2 + \eta \right] \le 2C_0.$$
According to the definition of $E_{\ell}$ in \eqref{E: epsilon_bar}, we have
$$ \mathcal{L}_\rho\left(x(t),\ell_\tau\right) - \mathcal{L}_\rho(x^*,\ell_\tau) \le \frac{2C_0}{t(\gamma t + r)},$$
for every $t\ge t_0$. This, combined with \eqref{E: Lagrangian_bar} and \eqref{E: lambda_tau}, results in
$$ \mathcal{L}_\rho\left(x(\tau),y^*\right) - \mathcal{L}_\rho(x^*,y^*) + \| Ax(\tau)-b \|
\le \frac{2C_0}{\tau(\gamma \tau + r)}. $$
Since $\tau\ge t_0$ was arbitrarily chosen, we conclude that
$$ \mathcal{L}_\rho\left(x(t),y^*\right) - \mathcal{L}_\rho(x^*,y^*) + \| Ax(t)-b \| 
\le \frac{2C_0}{t(\gamma t + r)}, $$
for every $t\ge t_0$, which is \eqref{E:main_thm}. Now, \eqref{E:main_thm} also shows that
$$ \| Ax(t)-b \| \le \frac{2C_0}{t(\gamma t + r)}$$
and
$$f\left( x(t) \right) - f(x^*) + \left\langle y^*, Ax(t) - b \right\rangle 
\le \frac{2C_0}{t(\gamma t + r)}.$$
Hence,
$$ f\left( x(t) \right) - f(x^*)  
\le \frac{2C_0}{t(\gamma t + r)} + \|y^* \| \left\| Ax(t) - b \right\| 
\le \frac{2C_0 \left(1+\|y^* \|\right)}{t(\gamma t + r)}. $$
On the other hand, by the convexity of $f$ and the optimality conditions, we have
\begin{align*}
f\left( x(t) \right) - f(x^*) 
&\ge \langle \nabla f(x^*), x(t) - x^* \rangle 
= - \langle A^* y^*, x(t) - x^* \rangle 
= - \langle y^*, Ax(t) - b \rangle.
\end{align*}
This results in
$$ f\left( x(t) \right) - f(x^*) 
\ge - \| y^* \|  \left\| Ax(t) - b \right\| 
\ge -\frac{2C_0 \| y^* \| }{t(\gamma t + r)},
$$
and completes the proof.
\end{proof}

\begin{remark}
The primal dual gap, the feasibility measure and the function values all enjoy a fast convergence rate $\mathcal{O}\left( \frac{1}{t^2} \right)$. More precisely, as $t\to\infty$, we have
\begin{align*}
\mathcal{L}_\rho\left(x(t),y^*\right) - \mathcal{L}_\rho(x^*,y^*)
&= \mathcal{O}\left( \tfrac{1}{\gamma t^2} \right),\quad
\|Ax(t)-b\|
= \mathcal{O}\left( \tfrac{1}{\gamma t^2} \right),\\
|f\left( x(t) \right) - f(x^*) |
&= \mathcal{O}\left( \tfrac{1}{\gamma t^2} \right).
\end{align*}
Setting $\gamma=1$, these rates match those in \cite{Radu_2021}. We also observe the time scaling property, as in \cite{Attouch_2022_ADMM,Hulett_2023}. While perserving fast convergence properties, the generated trajectories by \eqref{inertial_primal_dual_system} have less oscillations, thanks to the Hessian-driven damping $\beta\nabla^2_{xx}\mathcal{L}_\rho \left( x, y \right) \dot{x}$.  
\end{remark}

\subsection{Observing the stabilization effect}\label{Sec: example}
In this subsection, we consider the linearly constrained quadratic programming (QP) problem:
\begin{equation*}
\begin{aligned}
\min_{x\in\R^n}&\quad f(x) = \frac{1}{2}\langle x, P x\rangle , \\
\text{subject to}&\quad Ax=b,
\end{aligned}
\end{equation*}
where $P\in\R^{n\times n}$ is positive semidefinite, $A\in\R^{m\times n}$ and $b\in\R^m$. For simulation, we set 
\begin{equation*}
P = \begin{bmatrix}
10^{-2} & 0 & 0 \\
0 & 10 & 0\\
0 & 0 & 10^{2}
\end{bmatrix},\ 
A = \begin{bmatrix}
1 & 1 & 1
\end{bmatrix},\  
b = 1.
\end{equation*}
We also set $\alpha=3$, $\gamma=1$, $r=0.4$, $\rho=2$, $\theta = 0.5$, $x_0 = (3,3,3)$, $\dot{x}_0 = (0,0,0)$, $y_0 = 0$ and $\dot{y}_0 = 0$. 

We illustrate the effect of the Hessian-driven damping on the convergence behavior of \eqref{inertial_primal_dual_system}, by considering three cases: $\beta = 0$, $\beta = 0.1$ and $\beta = 0.4$. The results are shown in Figure \ref{fig: primal_dual_fx}, where the trajectory oscillations are alleviated with an increasing $\beta$. This confirms the role of the Hessian-driven damping in reducing the oscillations.  

\begin{figure}[htb!]
\centering
\includegraphics[width = 3.5in]{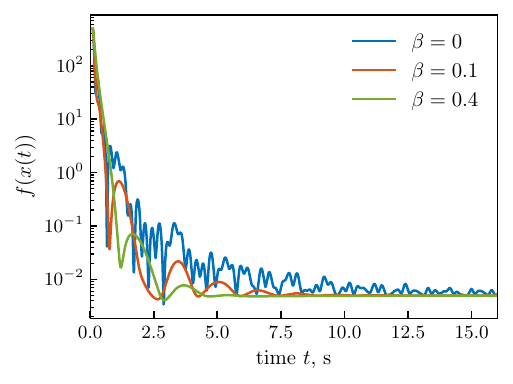}
\caption{With different values of $\beta$, the function values $f\big(x(t)\big)$ with respect to the time $t$.}
\label{fig: primal_dual_fx}
\end{figure}

\section{Conclusions}\label{Sec: conclusions}
In this paper, we propose two inertial primal dual dynamical systems for smooth and bilinearly coupled saddle point problems, and develop an inertial system for solving the affinely constrained convex optimization problem. These systems distinguish themselves by including a Hessian-driven damping term. For these three systems, we establish the trajectory boundedness, decaying property of the velocity, and fast convergence rates for the primal dual gap, especially in the strongly convex case, with or without knowledge of the strong convexity parameters.

\vskip 15pt

{\small
\noindent{\bf Acknowledgements.} This work was partially funded by the China Scholarship Council~202208520010, and also benefited from the support of the FMJH Program Gaspard Monge for optimization and operations research and their interactions with data science. \\

\noindent{\bf Compliance with Ethical Standards.} Both authors contributed equally to this research, and have no conflict of interest to declare. This research involved no human participants or animals, and did not make use of external data.
}

\bibliography{myrefs}

\end{document}